\documentclass{amsart}
\usepackage{mathrsfs}

\theoremstyle{plain}
\newtheorem{thm}{Theorem}[section]
\newtheorem{lemma}[thm]{Lemma}
\newtheorem{cor}[thm]{Corollary}
\newtheorem{prop}[thm]{Proposition}

\theoremstyle{definition}

\newtheorem{remark}[thm]{Remark}

\catcode`\@=11
\def\mequal{\mathrel{\mathpalette\@mvereq{\hbox{\sevenrm m}}}}
\def\@mvereq#1#2{\lower.5\p@\vbox{\baselineskip\z@skip\lineskip1.5\p@
    \ialign{$\m@th#1\hfil##\hfil$\crcr#2\crcr=\crcr}}}
\def\partr#1#2{/\kern-.08333em/_{#1,#2}^{\phantom{.}}}
\def\invpartr#1#2{/\kern-.08333em/_{#1,#2}^{-1}}
\def\hpartr#1#2{/\kern-.08333em/_{#1,#2}^{h}}
\def\Epartr#1#2{/\kern-.08333em/_{#1,#2}^{E}}

\def\newdot{{\kern.8pt\cdot\kern.8pt}}
\def\,{\relax\ifmmode\mskip\thinmuskip\else\thinspace\fi}
\def\{{\relax\ifmmode\lbrace\else $\lbrace$\fi}
\def\}{\relax\ifmmode\rbrace\else $\rbrace$\fi}

\font\sevenrm=cmr7

\newcommand\NN{\mathbb{N}}

\newcommand\RR{\mathbb{R}}

\newcommand{\SC}{{\mathscr C}}
\newcommand{\SD}{{\mathscr D}}
\newcommand{\SE}{{\mathscr E}}

\newcommand{\SF}{{\mathscr F}}

\newcommand{\SK}{{\mathscr K}}
\newcommand{\SL}{{\mathscr L}}

\newcommand{\ST}{{\mathscr T}}

\def\grad{\mathpal{grad}}

\def\grad{\mathop{\rm grad}\nolimits}
\def\di{\displaystyle}
\def\f{\frac}
\def\a{\alpha }
\def\b{\beta }

\def\d{\delta }

\def\G{\Gamma }
\def\g{\gamma }

\def\n{\nabla }

\begin{document}

\title[]{Medians and means in Finsler geometry}

\author[M. Arnaudon]{Marc Arnaudon} \address{Laboratoire de Math\'ematiques et
  Applications\hfill\break\indent CNRS: UMR 6086\hfill\break\indent
  Universit\'e de Poitiers, T\'el\'eport 2 - BP 30179\hfill\break\indent
  F--86962 Futuroscope Chasseneuil Cedex, France}
\email{marc.arnaudon@math.univ-poitiers.fr}

\author[F. Nielsen]{Frank Nielsen} \address{Laboratoire d'informatique (LIX)\hfill\break\indent
\'Ecole Polytechnique\hfill\break\indent 91128 Palaiseau Cedex,
France\hfill\break\indent
and\hfill\break\indent
Sony Computer Science Laboratories, Inc\hfill\break\indent
Tokyo, Japan} \email{Frank.Nielsen@acm.org}

%
%

\begin{abstract}\noindent
We investigate existence and uniqueness of $p$-means $e_p$ and the median $e_1$ of a probability measure $\mu$ on a Finsler manifold, in relation with the convexity of the support of $\mu$. We prove that $e_p$ is the limit point of a continuous time gradient flow. Under some additional condition which is always satisfied for $p\ge 2$, a discretization of this path converges to $e_p$.
This provides an algorithm for determining those Finsler center points. 
\end{abstract}

\maketitle
\tableofcontents

%
%

\section{Introduction}\label{Section1}
\setcounter{equation}0

The geometric barycenter of a set of points is the point which minimizes the sum of the {\it squared} distances to these points. 
It is the most traditional estimator is statistics that is however sensitive to outliers~\cite{breakdown}.
Thus it is natural to replace the average distance squaring (power~$2$) by taking the power of $p$ for some $p\in [1,2)$.
This leads to the definition of $p$-means. 
When $p=1$, the minimizer is the median of the set of points, very often used in robust statistics~\cite{breakdown}. 
In many applications, $p$-means with some  $p\in (1,2)$ give the best compromise. 
For existence and uniqueness in Riemannian manifolds under convexity conditions on the support of the measure, see Afsari~\cite{Afsari:10}.

The Fermat-Weber problem concerns finding the median $e_1$ of a set of points in an Euclidean space.  
Numerous authors worked out algorithms for computing~$e_1$. 
The first algorithm was proposed by Weiszfeld in~\cite{Weiszfeld:37} (see also~\cite{Vardi:00}). 
It has been extended to sufficiently small domains in Riemannian manifolds with nonnegative curvature by Fletcher and al. in~\cite{Fletcher:09}. 
A complete generalization to manifolds with positive or negative curvature (under some convexity conditions in positive curvature), has been recently given by Yang in~\cite{Yang:10}.

The Riemannian barycenter or Karcher mean of a set of points in a manifold or more generally of a probability measure has been extensively studied, see e.g. \cite{Karcher:77}, \cite{Kendall:90}, \cite{Kendall:91}, \cite{Emery-Mokobodzki:91}, \cite{Picard:94}, \cite{Arnaudon-Li:05}, \cite{Corcuera-Kendall:99}, where questions of existence, uniqueness, stability, relation with martingales in manifolds, behavior  when measures are pushed by stochastic flows have been considered. The Riemannian barycenter corresponds to $p=2$ in the above description. Computation of Riemannian barycenters by gradient descent has been performed by Le in \cite{Le:04}.

The aim of this paper is to extend to the context of Finsler manifolds the results on existence and uniqueness of $p$-means of probability measures, as well as algorithms for computing them. Some convexity is needed, and as we shall see the fact that comparison results for triangles as Alexandroff and Toponogov theorems do not exist impose more restrictions on the support of the probability measure. As a consequence, the sharp results on existence and uniqueness established by Afsari~\cite{Afsari:10} and the algorithm for computing means of Yang in~\cite{Yang:10} do not extend to Finsler manifolds.

The motivation for this work primarily comes from signal filtering and denoising in the context of Diffusion Tensor Imaging (DTI), High Angular Resolution Imaging (HARDI, see~\cite{Shen:06}, \cite{PKD:09}, \cite{Astola-Florack:09}), Orientation Distribution Function (ODF), active contours~\cite{MPAT:08}. 
Applications with experimental results of an implementation will be reported in forthcoming papers.

Information geometry at its heart considers the differential geometry
nature of probability distributions induced by a divergence function.
In probability theory, invariance by monotonic re-parameterization and
sufficient statistics yields  the class of
$f$-divergences~\cite{informationgeometry-2000} $I_f(p,q)=\int p(x)
f(\frac{q(x)}{p(x)}) \mathrm{d}x$ that includes the Kullback-Leibler (KL)
information-theoretic  divergence $\mathrm{KL}(p,q)=\int p(x)
\log\frac{p(x)}{q(x)} \mathrm{d}x$ as its prominent member (for
$f(t)=-\log t$).
It is well-known that the KL divergence (better known as the relative
entropy) yields a dually flat structure~\cite{informationgeometry-2000}
generalizing the (self-dual) Euclidean space.

Because divergences are usually asymmetric and violate the triangle
inequality they have not been extensively considered from an {\it
algorithmic} point of view.
Indeed, the triangle inequality property is often used in computational
geometry to design {\it efficient} algorithms by allowing various
``pruning'' techniques~\cite{Elkan:2003,MetricSkyline:2009}. Computational
geometry has thus mostly considered metric spaces for keeping the triangle
inequality properties.

One can metrize divergences.
The KL divergence can be symmetrized either into the Jeffreys divergence
$J(p,q)=\mathrm{KL}(p,q)+\mathrm{KL}(q,p)$ or the Jensen-Shannon (JS)
divergence:
\begin{align*}\mathrm{JS}(p,q)&=\mathrm{KL}(p,\frac{p+q}{2})+\mathrm{KL}(q,\frac{p+q}{2})\\&=\int
(p(x)\log \frac{2p(x)}{p(x)+q(x)} + q(x)\log \frac{2q(x)}{p(x)+q(x)} )
\mathrm{d}x .\end{align*}
The latter is preferred in practice because it is bounded, and its square
root yields a metric that can be embedded into a Hilbert
space~\cite{FugledeTopsoe:2004}.

Finsler distances, arising from the underlying the Finsler metrics, are
attractive as they preserve the triangle inequality~\cite{Shen:06} for
efficient algorithmics but potentially model asymmetric distances.

In information geometry, the regular divergence $D$ associated to a
Finslerian metric distance $d$ can be defined as $D(p,q)=d^2(p,q)$.
Observe that the Finslerian-based divergence looses then the triangle
inequality property~\cite{Shen:06}.
(eg., the squared Euclidean distance does not satisfy the triangle
inequality).

\section{Preliminaries}\label{Section2}
\setcounter{equation}0

Let $M$ be a smooth manifold. On $M$ we consider a Finsler
structure $F : TM\to \RR_+$. For any $x\in M$, $V,X,Y,Z\in T_xM$
such that $V\not=0$, let
\begin{equation}
\label{1}
g_V(X,Y):=\f12\f{\partial^2}{\partial s\partial
t}\Big|_{(s,t)=(0,0)}F^2(V+sX+tY).
\end{equation}
(we shall also use the notation ${<}X,Y{>}_V=g_V(X,Y)$) and
\begin{equation}
\label{2}
{<}X,Y,Z{>}_V:=\f14\f{\partial^3}{\partial r\partial s\partial
t}\Big|_{(r,s,t)=(0,0,0)}F^2(V+rX+sY+tZ).
\end{equation}
We have
\begin{equation}
\label{3}
{<}X,Y,Z{>}_V=\f12\f{\partial}{\partial r}\Big|_{r=0}g_{V+rX}(Y,Z)
\end{equation}
and in particular since $F^2$ is $2$-homogeneous and $V\mapsto
g_V(X,Y)$ is $0$-homogeneous,
\begin{equation}
\label{4}
{<}V,Y,Z{>}_V=0.
\end{equation}
Let $V$ be a non-vanishing vector field on $M$. The Chern
connection $\n^V$ is torsionfree and almost metric, 
and can be characterized by 
\begin{equation}
\label{5}
X{<}Y,Z{>}_V={<}\n^V_XY,Z{>}_V+{<}Y,\n^V_XZ{>}_V+2{<}\n^V_XV,Y,Z{>}_V.
\end{equation}
More precisely, parameterizing locally $TM$ by coordinates $$(x^1,\ldots,x^m,y^1=dx^1,\ldots,y^m=dx^m),$$ defining the geodesic coefficients as 
\begin{equation}
\label{25}
G^i(y)=\f14g^{ik}(y)\left(2\f{\partial g_{jk}}{\partial x^l}-\f{\partial g_{jl}}{\partial x^k}\right)y^jy^l,\quad y\in TM\backslash\{0\},
\end{equation}
letting
\begin{equation}
\label{26}
N_j^i=\f{\partial G^i}{\partial y^j},\quad \f{\d}{\d x^i}=\f{\partial}{\partial x^i}-N_i^k(y)\f{\partial }{\partial y^k}\in T_y\left(TM\backslash\{0\}\right),
\end{equation}
then the Christoffel symbols of the Chern connection are given by 
\begin{equation}
\label{27}
\G_{ij}^k=\f12 g^{kl}\left(\f{\d g_{lj}}{\d x^i}+\f{\d g_{il}}{\d x^j}-\f{\d g_{ij}}{\d x^l}\right)
\end{equation}
(see \cite{Chern-Shen:05}).
Note that defining 
\begin{equation}
\label{28}
\d y^i=dy^i+N_j^i(y)dx^j
\end{equation}
we have for a smooth function $f : TM\backslash\{0\}\to \RR$
\begin{equation}
\label{29}
df =\f{\d f}{\d x^i}dx^i+\f{\partial f}{\partial y^i}\d y^i.
\end{equation}

The Chern curvature
tensor is defined by the equation
\begin{equation}
\label{6}
R^V(X,Y)Z:=\n^V_X\n^V_YZ-\n^V_Y\n^V_XZ-\n^V_{[X,Y]}Z,
\end{equation}
and the flag curvature is 
\begin{equation}
\label{7}
\SK(V,W):=\f{{<}R^V(V,W)W,V{>}}{{<}V,V{>}_V{<}W,W{>}_V-{<}V,W{>}_V^2},
\end{equation}
for two non collinear $V,W\in T_xM$.

We say that $M$ has nonpositive flag curvature if for all $V,W$, $\SK(V,W)\le 0$. 

The tangent curvature of two vectors $V,W\in T_xM$ is defined as 
\begin{equation}
\label{39}
\ST_V(W)={<}\n_{ W}^{ W}{\tilde W}-\n_{ W}^{ V}{\tilde W},{ V}{>}_V
\end{equation}
where ${\tilde W}$ is a vector field satisfying  ${\tilde W}_x=W$. For a nonnegative constant $\d\ge 0$ we say that $\ST \ge -\d$ or $\ST\le \d$ if respectively
\begin{equation}
\label{40}
\ST_V(W)\ge -\d F(V)F(W)^2\quad\hbox{or}\quad \ST_V(W)\le \d F(V)F(W)^2.
\end{equation}

For $x\in M$ we define 
\begin{equation}
 \label{46}
\SC(x)=\sup_{v,w\in T_xM\backslash\{0\}}\sqrt{\f{{<}v,v{>}_v}{{<}v,v{>}_w}}, \quad \SD(x)=\sup_{v,w\in T_xM\backslash\{0\}}\sqrt{\f{{<}v,v{>}_w}{{<}v,v{>}_v}}.
\end{equation}

\begin{remark}
 For applications in active contours, a ``Wulff shape'' is given which does not depend on $x$ and defines the Finsler structure. From this shape $\SC$ and $\SD$ can easily be calculated. See e.g \cite{MPAT:08} and \cite{ZSN:09}.
\end{remark}

A geodesic in $M$ is a curve $t\mapsto c(t) $ satisfying for all $t$,  $\di \n_{\dot c(t)}^{\dot c(t)}\dot c=0$. It is well known that a geodesic has constant speed, and that it locally minimizes the distance (\cite{Chern-Shen:05}). If so, letting $\rho(x,y)$ the forward distance from $x$ to $y$, then 
\begin{equation}
\label{8}
\rho^2(x,y)={<}\dot c(0), \dot c(0){>}_{\dot c(0)}
\end{equation}
where $t\mapsto c(t)$ is the minimal geodesic satisfying $c(0)=x$ and $c(1)=y$. By definition, the backward distance from $x$ to $y$ is $\rho(y,x)$.

 For $x\in M$ and $v\in T_xM$, we let whenever it exists $\exp_x(v):=c(1)$ where $t\mapsto c(t)$ is the geodesic satisfying $\dot c(0)=v$.

If $M$ is complete, analytic, simply connected and has nonpositive flag curvature (we say that $M$ is an analytic Cartan-Hadamard manifold), then $\exp_x : T_xM\to M$ is an homeomorphism (\cite{Auslander:55} theorem 4.7).  Under these assumption, letting for $x,y\in M$,  $\overrightarrow{xy}=\exp_x^{-1}(y)$, we have 
\begin{equation}
\label{9}
\rho^2(x,y)={<}\overrightarrow{xy}, \overrightarrow{xy}{>}_{\overrightarrow{xy}}.
\end{equation}

For $x_0\in M$ and $R>0$,, let us denote by $B(x_0,R)$ (resp. $\bar B(x_0,R)$) the (forward) open (resp. closed) ball  with center $x_0$ and radius $R$: 
\begin{equation}
\label{41}
B(x_0,R)=\{y\in M, \ \rho(x_0,y)<R\}\quad (\hbox{resp.}\quad \bar B(x_0,R)=\{y\in M, \ \rho(x_0,y)\le R\})
\end{equation}

Now let $(t,s)\mapsto c(t,s)$ a family of minimizing geodesics $t\mapsto c(t,s)$, $t\in [0,1]$,  parametrized by $s \in I$, $I$ an interval in $\RR$. Define 
\begin{equation}
\label{10}
E(s)=\f12 \rho^2(c(0,s),c(1,s)).
\end{equation}
The computation of $E'(s)$ and $E''(s)$ is well-known, see e.g. \cite{Bao-Chern-Shen:00}. We have
\begin{equation}
\label{11}
E'(s)={<}\partial_s c(1,s), \partial_t c(1,s){>}_{\partial_t c(1,s)}-{<}\partial_s c(0,s), \partial_t c(0,s){>}_{\partial_t c(0,s)}.
\end{equation}
As for the second derivative, letting   $c=c(\cdot, 0)$, and for $X,Y$ vector fields along $c$
\begin{equation}
\label{12}
I(X,Y)=\int_0^1\left({<}\n_{T}^{T}X, \n_{T}^{T}Y{>}_{T}
-{<}R^T(X,T)T,Y{>}_{T}\right)\,dt
\end{equation}
the index of $X$ and $Y$, we have 
\begin{equation}
\label{13}
\begin{split}
E''(0)&={<}\n_{\partial_sc(1,0)}^{\partial_tc(1,0)}\partial_sc(1,\cdot), \partial_tc(1,0){>}_{\partial_tc(1,0)}-{<}\n_{\partial_sc(0,0)}^{\partial_tc(0,0)}\partial_sc(0,\cdot), \partial_tc(0,0){>}_{\partial_tc(0,0)}\\
&+I(\partial_sc(\cdot, 0),\partial_sc(\cdot, 0)).
\end{split}
\end{equation}
Assuming $s\mapsto c(0,s)$ and $s\mapsto c(1,s)$ are geodesics, we obtain
\begin{equation}
\label{42}
\begin{split}
E''(0)&=\ST_{\partial_tc(0,0)}(\partial_sc(0,0))-\ST_{\partial_tc(1,0)}(\partial_sc(1,0))+I(\partial_sc(\cdot, 0),\partial_sc(\cdot, 0)).
\end{split}
\end{equation}

We are interested in the situation where $c(1,s)\equiv z$ a constant. In this case we have
\begin{equation}
\label{14}
\begin{split}
E''(0)&=\ST_{\partial_tc(0,0)}(\partial_sc(0,0))+I(\partial_sc(\cdot, 0),\partial_sc(\cdot, 0)).
\end{split}
\end{equation}
For $p\ge 1$, define 
\begin{equation}
 \label{44}
D_p(s)=\rho^p(c(0,s), z)
\end{equation}

\begin{prop}
\label{P8}
Assume $\SK\le k$, $\ST\ge -\d$, $\SC\le C$ for some $k,\d \ge 0$, $C\ge 1$. Let $p> 1$.
Then writing $r=\rho(x,z)$,
\begin{equation}
 \label{73}
D_p''(0)\ge pr^{p-2}\left(\min\left(p-1,\f{\sqrt kr\cos (\sqrt kr)}{\sin (\sqrt kr)} \right)C^{-2}-\d r\right).
\end{equation}
 If $z$ and $x=c(0,0)$ satisfy $\rho(x,z)< R(p,k,\d, C)$ with 
\begin{equation}
\label{43}
 R(p,k,\d, C)=\min\left(\f{p-1}{C^2\d}, \f1{\sqrt{k}}\arctan\left(\f{\sqrt{k}}{C^2\d}\right)\right)
\end{equation}
and the injectivity radius at $x$ is strictly larger than $R(p,k,\d, C)$, 
 then $D_p''(0)> 0$.
\end{prop}
\begin{remark}
 Note if $p\ge 2$ then 
$$
R(p,k,\d, C)=R(2,k,\d, C)=\f1{\sqrt{k}}\arctan\left(\f{\sqrt{k}}{C^2\d}\right).
$$
\end{remark}

\begin{proof}
Define $T(t)=\partial_tc(t,0)$,  $J(t)=\partial_sc(t,0)$, $$ J^T(t)=\f1{F(T(t))^2}{<}J(t),T(t){>}_{T(t)}T(t),\quad J^N(t)=J(t)-J^T(t).$$
Using successively \cite{Bao-Chern-Shen:00} Lemma~9.5.1 which compares the index $I(J,J)$ with the one of its ``transplant'' into a manifold with constant curvature $k^2$ and the index lemma \cite{Bao-Chern-Shen:00} Lemma~7.3.2 which compares the index of the transplant to the one of the Jacobi field with same boundary values, we get, letting $r=\rho(x,z)=D_1(0)$, 
\begin{equation}
 \label{45}
I(J,J)\ge \f{\sqrt kr\cos (\sqrt kr)}{\sin (\sqrt kr)}{<}J^N(0),J^N(0){>}_{T(0)}+ {<}J^T(0),J^T(0){>}_{T(0)}.
\end{equation}
Using the expression~\eqref{14} for $E''(0)$ we obtain
\begin{equation}
 \label{47}
E''(0)\ge -\d r+ \f{\sqrt kr\cos (\sqrt kr)}{\sin (\sqrt kr)}{<}J^N(0),J^N(0){>}_{T(0)}+ {<}J^T(0),J^T(0){>}_{T(0)}.
\end{equation}
We have from~\eqref{11}
\begin{equation}
 \label{48}
E'(0)^2= r^2{<}J^T(0),J^T(0){>}_{T(0)}.
\end{equation}

Now from $D_1(s)=\sqrt{2E(s)}$ we get
\begin{equation}
\label{47bis}
D_1'(s)=\f{E'(s)}{D_1(s)},\quad D_1''(s)=\f{E''(s)}{D_1(s)}-\f{E'(s)^2}{D_1^3(s)},
\end{equation}
and this yields 
\begin{align*}
 &D_p''(0)\\&=pD_1(0)^{p-2}\left((p-1)D_1'(0)^2+D_1(0)D_1''(0)\right)\\
&=pr^{p-2}\left((p-2){<}J^T(0),J^T(0){>}_{T(0)}+E''(0)\right)\\
&\ge pr^{p-2}\left((p-1){<}J^T(0),J^T(0){>}_{T(0)}-\d r+ \f{\sqrt kr\cos (\sqrt kr)}{\sin (\sqrt kr)}{<}J^N(0),J^N(0){>}_{T(0)}\right)\\
&\ge pr^{p-2}\left(\min\left(p-1,\f{\sqrt kr\cos (\sqrt kr)}{\sin (\sqrt kr)} \right){<}J(0),J(0){>}_{T(0)}-\d r\right)\\
&\ge pr^{p-2}\left(\min\left(p-1,\f{\sqrt kr\cos (\sqrt kr)}{\sin (\sqrt kr)} \right)C^{-2}-\d r\right).
\end{align*}
>From this bound the rest of the proof follows easily.
\end{proof}

Similarly, we can obtain an upper bound for $D_p''(0)$:
\begin{prop}
 \label{P6}
Assume the sectional curvatures $\SK$ have a lower bound $-\b^2$ for some $\b>0$, and  $\ST \le \d'$ for some $\d'>0$, $\SD\le D$ for some $D\ge 1$. Again let $r=\rho(x,z)$, assume that the injectivity radius at $x$ is larger than~$r$. 
Then 
\begin{equation}
 \label{54}D_p''(0)\le 
pr^{p-2}\left(D^2\max\left(p-1,\b r\coth (\b r) \right)+\d' r\right).
\end{equation}
\end{prop}
\begin{proof}
We have by~\eqref{14} and~\eqref{47bis} together with the fact that 
\begin{equation}
 I(J,J)={<}J^T(0),J^T(0){>}+I(J^N,J^N),
\end{equation}
\begin{align*}
 D_p''(0)=pr^{p-2}\left(p-1){<}J^T(0),J^T(0){>}_{T(0)}+I(J^N,J^N)+\ST_T(J)\right).
\end{align*}
Let $t\mapsto X(t)$ the parallel vector field along $t\mapsto c(t,0)$ with initial condition $J^N(0)$, and for $t\in [0,1]$, let 
\begin{align*}
 G(t)=\cosh(r\b t)-\coth\left({r\b}\right)\sinh(r\b t).
\end{align*}
This is the solution of $G''=r\b G$ with conditions $G(0)=1$ and $G(1)=0$.
The vector field $t\mapsto Y(t)$ along $t\mapsto c(t,0)$ defined by 
\begin{equation}
 \label{56}
Y(t)=G(t)X(t)
\end{equation}
has same boundary values as $t\mapsto J^N(t)$, so by the index lemma \cite{Bao-Chern-Shen:00} Lemma~7.3.2 we have 
\begin{equation}
 \label{57}
I(J^N,J^N)\le I(Y,Y).
\end{equation}
On the other hand 
\begin{align*}
& I(Y,Y)\\&=\int_0^1\left(G'(t)^2{<}J^N(0),J^N(0){>}_{T(0)}-G(t)^2{<}R^T(X(t),T(t))T(t),X(t){>}_{T(t)}\right)\,dt\\
&\le {<}J^N(0),J^N(0){>}_{T(0)}\int_0^1\left(G'(t)^2+r^2\b^2G(t)^2\right)\,dt\\
&={<}J^N(0),J^N(0){>}_{T(0)}\left(\left[G'(t)G(t)\right]_0^1+\int_0^1G(t)\left(-G''(t)^2+r^2\b^2G(t)\right)\,dt\right)\\
&={<}J^N(0),J^N(0){>}_{T(0)}r\b\coth\left({r\b}\right).
\end{align*}
So 
\begin{equation}
 \label{71}
\begin{split}
&D_p''(0)\\
&\le pr^{p-2}\left((p-1){<}J^T(0),J^T(0){>}_{T(0)}+r\b\coth (r\b){<}J^N(0),J^N(0){>}_{T(0)}\right)+\d'r\\
&\le pr^{p-2}\left(\max \left((p-1),r\b\coth (r\b)\right){<}J(0),J(0){>}_{T(0)}+\d'r\right)\\
&\le pr^{p-2}\left(D^2\max \left((p-1),r\b\coth (r\b)\right)+\d'r\right)
\end{split}
\end{equation}
since $F(J(0))=1$. 

\end{proof}

\bigbreak

For $x\in M$, let $\ell_x : T_xM\to T_x^\ast M$ be the Legendre transformation, defined as 
\begin{equation}
\label{17}
\ell_x(V)=g_V(V,\cdot)\quad\hbox{if}\quad V\not=0,\quad \ell_x(0)=0.
\end{equation}
It is well-known that $\ell_x$ is a bijection.
The global Legendre transformation on $TM$ is defined as 
\begin{equation}
\label{18}
\SL(V)=\ell_{\pi(V)}(V)
\end{equation}
where $\pi : TM\to M$ is the canonical projection.
 If we define the dual Minkowski norm $F^\ast$ on $T_x^\ast M$ as 
\begin{equation}
\label{19}
F^\ast(\xi)=\max\{\xi(y),\ y\in T_xM,\ F(y)=1\},
\end{equation}
then 
\begin{equation}
\label{20}
F=F^\ast\circ \SL
\end{equation}
and for non zero $V\in TM$ and $\a\in T^\ast M$, 
\begin{equation}
\label{20bis}
\langle \SL(V),V\rangle=F(V)^2,\quad \quad
\langle \a, \SL^{-1}(\a)\rangle =F^\ast (\a)^2
\end{equation}
(see e.g. \cite{Alvarez_Paiva:06}).

For $f$ a $C^1$ function on $M$ we may define the gradient of $f$ 
\begin{equation}
\label{21}
\grad f=\SL^{-1}(df).
\end{equation}

\section{Forward $p$-means }\label{Section3}
\setcounter{equation}0

Let $\mu$ be a compactly supported probability measure in $M$. For $p>1$ and  $x\in M$ we define 
\begin{equation}
\label{16}
\SE_{\mu,p}(x)=\int_M\rho^p(x,z)\,\mu(dz).
\end{equation}
The (forward) $p$-mean of $\mu$ is the point~$e_p$ of $M$ where $\SE_{\mu,p}$ reaches its minimum whenever it exists and is unique.  

In this paper we will consider forward $p$-means and we will call them $p$-means. Similarly we could define the backward $p$-mean $\overleftarrow e_p$ as the point which minimizes
$$
x\mapsto \overleftarrow\SE_{\mu,p}(x):=\int_M\rho^p(z,x)\,\mu(dz).
$$
 Depending on the context, forward or backward mean is more appropriate. One should note that defining the reverse (or adjoint) Finsler  structure  $\overleftarrow F (v)=F(-v)$, $v\in TM$, it is easy to check that the associated distance $\overleftarrow \rho$ satisfies $\overleftarrow \rho(z,x)=\rho(x,z)$, and forward $p$-mean  for $\overleftarrow F$ is backward $p$-mean for $F$. So without loss
of generality we can consider only the forward $p$-means.

 One should also note that in High Angular Resolution Imaging the Finsler structure is symmetric, so both notions coincide. It is not the case for the application concerning active contours where it is natural to consider non symmetric $F$. 

Even if it is in a non-Finslerian context, one can give the example of right-sided and left-sided Kullback-Leibler divergences
 for families of Gaussian probability densities (see \cite{Nielsen-Nock:09}). 
The left sided centroid focuses on the highest mode (it is zero-forcing), and
 the right-sided centroid tries to cover the support of both normals (it is zero-avoiding as depicted in Fig.2 of~\cite{Nielsen-Nock:09}).

\begin{prop}
\label{P1}
Assume there exists $C>0$ such that $\SC(x)\le C$ for all $x\in M$, where $\SC(x)$ is defined in~\eqref{46}. Assume furthermore that 
 ${\rm supp}(\mu)\subset B(x_0,R)$ for some $x_0\in M$ and $R>0$. 
Then $x\mapsto \SE_{\mu,p}(x)$ has at least one global minimum in $\bar B(x_0, C(1+C)R)$.
\end{prop}
\begin{proof}
 We begin with establishing that for all $y_1,y_2\in M$, 
\begin{equation}
 \label{49}
\f1C\rho(y_2,y_1)\le \rho(y_1,y_2)\le C\rho(y_2,y_1).
\end{equation}
It is sufficient to establish the second inequality and then to exchange $y_1$ and $y_2$. If $t\mapsto \varphi(t)$ is a path from $y_1=\varphi(0)$ and $y_2=\varphi(1)$ then its length  $L(\varphi)$ satisfies 
\begin{align*}
L(\varphi)&=\int_0^1 \sqrt{{<}\dot\varphi(t),\dot\varphi(t){>}_{\dot\varphi(t)}}\,dt\\
&= \int_0^1 \sqrt{\f{{<}-\dot\varphi(t),-\dot\varphi(t){>}_{\dot\varphi(t)}}{{<}-\dot\varphi(t),-\dot\varphi(t){>}_{-\dot\varphi(t)}}}\sqrt{{<}-\dot\varphi(t),-\dot\varphi(t){>}_{-\dot\varphi(t)}}\,dt\\
&\le \int_0^1 \SC(\varphi(t))\sqrt{{<}-\dot\varphi(t),-\dot\varphi(t){>}_{-\dot\varphi(t)}}\,dt\\
&\le C\int_0^1 \sqrt{{<}-\dot\varphi(t),-\dot\varphi(t){>}_{-\dot\varphi(t)}}\,dt\\
&=CL(\hat \varphi)
\end{align*}
where $\hat\varphi$ is the path from $y_2 $ to $y_1$ defined by $\hat\varphi(t)=\varphi(1-t)$. Minimizing over all paths $\hat\varphi$ from $y_2$ to $y_1$ we get 
\begin{equation}
 \label{50}
\rho(y_1,y_2)\le C\rho(y_2,y_1).
\end{equation}

Now if ${\rm supp}(\mu)\subset B(x_0,R)$ then  $\SE_{\mu,p}(x_0)\le R^p$. On the other hand, if $x\not\in \bar B(x_0,C(1+C)R)$ then for all $y\in B(x_0,R)$
\begin{align*}
\rho(x,y)&\ge  \rho(x,x_0)-\rho(y,x_0)\\
&\ge \f1{C}\rho(x_0,x)-C\rho(x_0,y)\\
&\ge (1+C)R-CR=R
\end{align*}
and this clearly implies that $\SE_{\mu,p}(x)\ge R^p$. From this we get the conclusion.

\end{proof}

Concerning the uniqueness of the global minimum of $\SE_{\mu,p}$, we also have the following easy result.

\begin{prop}
\label{P2}
Assume that $\mu$ is supported by a compact ball $\bar B(x_0,R)$, and  that for all $z\in\bar B(x_0,R)$, the function $x\mapsto \rho^p(x,z)$  is strictly convex in $\bar B(x_0,C(1+C)R)$. Then $\mu$ has a unique  $p$-mean in $\bar B(x_0,C(1+C)R)$.
\end{prop}

\begin{proof}
If $x\mapsto\rho^p(x,z)$ is strictly convex for all $z$ in the support of $\mu$ then $\SE_{\mu,p}$ is strictly convex, and this implies that it has a unique minimum, which is attained at a unique point $e_p$.
\end{proof}

\begin{cor}
\label{C1}
Assume $\SK\le k$, $\ST\ge -\d$, $\SC\le C$ for some $k,\d \ge 0$, $C\ge 1$. Let $p>1$. Again let 
$$
R(p,k,\d,C)=\min\left(\f{p-1}{C^2\d}, \f1{\sqrt{k}}\arctan\left(\f{\sqrt{k}}{C^2\d}\right)\right)
$$
If $\mu$ is supported by a  geodesic ball $B(x_0,R)$ with 
\begin{equation}
\label{51}
 R\le \f1{C(C+1)^2}R(p,k,\d,C)
\end{equation}
and the injectivity radius at any $x\in B(x_0,C(1+C)R)$ is strictly larger than $R(p,k,\d,C)$
then $\mu$ has a unique $p$-mean 
$e_p$ satisfying 
\begin{equation}
 \label{52}
e_p\in \bar B\left(x_0,\f1{C+1}R(p,k,\d,C)\right).
\end{equation}

\end{cor}

\begin{proof}
If $x,z\in B(x_0, C(1+C)R)$ then $$\rho(x,z)\le \rho(x,x_0)+\rho(x_0,z)\le (1+C)^2CR\le R(p,k,\d,C). $$
Using proposition~\ref{P8}, we obtain that $\SE_{\mu,p}$ is strictly convex on $B(x_0, C(1+C)R)$. So by proposition~\ref{P2} $\mu$ has a unique $p$-mean in $\bar B(x_0, C(1+C)R)$.
\end{proof}
\begin{remark}
 \label{R1}
Letting $x_0\in M$, $D$ be a relatively compact neighborhood of $x_0$, then $\SK$ and $\SC$ are bounded above  on $D$ by, say $k_D$ and $C_D$, and $\ST$ is bounded below on $D$ by $-\d_D$. Using these bounds instead of $k$, $C$ and $\d$, we can find $R$ sufficiently small so that the conditions of corollary~\ref{C1} are fulfilled. So we can say any measure $\mu$ with sufficiently small support has a unique $p$-mean.
\end{remark}
\begin{remark}
 \label{R3}
If $M$ is a Cartan-Hadamard manifold, we recover the fact that we can take $R(p,k,\d,C)$ as large as we want.

More generally, in the Riemannian case,  Afsari~\cite{Afsari:10} proved existence and uniqueness of $p$-means, $p\ge 1$ on geodesic balls with radius $\di r<\f12\min\left\{{\rm inj}(M),\f{\pi}{2\sqrt{k}}\right\}$ if $p\in [1,2)$, and $\di r<\f12\min\left\{{\rm inj}(M),\f{\pi}{\sqrt{k}}\right\}$ if $p\ge 2$. Even taking $\d=0$ and $C=1$ in Corollary~\ref{C1} the support of $\mu$ has half the size of the one in~\cite{Afsari:10}  for $p\in (1,2)$ due to the fact that  we have an additional condition~\eqref{51} coming from the non optimality of Proposition~\ref{P2} in the Riemannian context. As for $p\ge 2$ another factor two is gained in~\cite{Afsari:10} with repeated use of Toponogov and Alexandroff theorems which are not available in our context.
\end{remark}
\begin{remark}
 \label{R3bis}
The condition on injectivity radius is the same as in the Riemannian case. The cut locus of any  point of $x\in B(x_0,C(1+C)R)$ has to be at distance larger than $R(p,k,\d,C)$. As for Riemannian manifold there is no general condition which insures this property for cut points, but for conjugate points the same condition holds, due to Rauch comparison theorem, see Theorem~9.6.1 in~\cite{Bao-Chern-Shen:00}. In the particular case when $M$ is a Cartan-Hadamard Finsler manifold, i.e. it has nonpositive flag curvature and it is simply connected, then the injectivity radius in everywhere infinite (see Theorem~9.4.1 in~\cite{Bao-Chern-Shen:00}).
\end{remark}

\begin{prop}
\label{P3}
Let  $a\mapsto x(a)$ solve the equation 
\begin{equation}
\label{22}
x(0)=x_0\quad \hbox{and for $a\ge 0$}\quad x'(a)=\grad_{x(a)}(-\SE_{\mu,p}).
\end{equation}
Under the conditions of Corollary~\ref{C1}, the path $a\mapsto x(a)$ converges as $a\to \infty$ to the $p$-mean of~$\mu$.
\end{prop}
\begin{proof}
If $f(a)=(-\SE_\mu)(x(a))$ we have as soon as $\grad_{x(a)}(-\SE_{\mu,p})\not=0$, 
\begin{align*}
f'(a)&=\left\langle d_{x(a)}(-\SE_{\mu,p}), x'(a)\right\rangle\\
&=\left\langle d_{x(a)}(-\SE_{\mu,p}), \grad_{x(a)}(-\SE_{\mu,p})\right\rangle\\
&=\left\langle d_{x(a)}(-\SE_{\mu,p}), \SL^{-1}(d_{x(a)}(-\SE_{\mu,p}))\right\rangle\\
&=F^\ast\left(d_{x(a)}(-\SE_{\mu,p})\right)^2
\end{align*}
by \eqref{20} and \eqref{20bis}.

On the other hand, we have $f(0)\ge -R^p$ and $f$ is nondecreasing. This implies that for all $a\ge 0$, $x(a)\in \bar B(x_0,C(1+C)R)$, since for all $x\not\in \bar B(x_0,C(1+C)R)$, $\SE_{\mu,p}(x)\ge R^p$. As a consequence $x(a)$ has limit points as~$a$ goes to infinity, and since $f(a)$ converges, any limit point is a critical point of $x\mapsto \SE_{\mu,p}(x)$. But by Proposition~\ref{P2} $\SE_{\mu,p}$ has a unique critical point in $\bar B(x_0,C(1+C)R)$ which is the mean~$e_p$ of~$\mu$. So we can conclude that $x(a)$ converges to~$e_p$.
\end{proof}

\section{Forward median }\label{Section4}
\setcounter{equation}0

Let $\mu$ be a compactly supported probability measure in $M$. For $x\in M$ we define 
\begin{equation}
\label{16bis}
\SF_\mu(x)=\int_M\rho(x,z)\,\mu(dz).
\end{equation}
The  median of $\mu$ is the point in $M$ where $\SF_\mu$ reaches its minimum whenever it exists and is unique.  

Again we have the following result.
\begin{prop}
\label{P4}
Assume there exists $C>0$ such that $\SC(x)\le C$ for all $x\in M$. Assume furthermore that 
 ${\rm supp}(\mu)\subset B(x_0,R)$ for some $x_0\in M$ and $R>0$. 
Then $x\mapsto \SF_\mu(x)$ has at least one global minimum in $\bar B(x_0, C(1+C)R)$.
\end{prop}
\begin{prop}
\label{P5}
Assume that $\mu$ is supported by a compact  ball $\bar B(x_0,R)$, that the support of $\mu$ is not contained in a single geodesic and  that for all $z\in\bar B(x_0,R)$, the forward distance to $z$ is convex, and strictly convex in any geodesic of $\bar B(x_0,C(1+C)R)$ which does not contain $z$. Then $\mu$ has a unique median $m\in\bar B(x_0,C(1+C)R)$.
\end{prop}
\begin{proof}
Clearly under these assumptions $\SF_\mu$ is strictly convex, so it has a unique local minimum, this minimum is global and is attained at a unique point $m\in\bar B(x_0,C(1+C)R)$.
\end{proof}

\begin{remark}
 \label{R2}
Contrarily to the case of $p$-means for $p>1$, we cannot say at this stage that any probability measure $\mu$ with sufficiently small support has a unique median, since we don't know whether $\SF_\mu$ is strictly convex or not. In the next proposition we give a sufficient condition for strict convexity of $\SF_\mu$.
\end{remark}

\begin{prop}
 \label{P9}
Assume $\SK\le k$ and $\ST\ge -\d$ for some $k,\d>0$.
Assume that the injectivity radius at any point of $\bar B(x_0, C(1+C)R)$ is larger than $(C^2+C+1)R$.
 Define 
\begin{equation}
\begin{split}
 \label{58}
\eta=&\min\Bigl\{\int_M\sqrt{k}\cot\left(\sqrt{k}\rho(\pi(v),z)\right){<}v^N,v^N{>}_{\overrightarrow{\pi(v)z}}\,\mu(dz),\\&\ \qquad v\in TM\ \hbox{satisfying}\ \pi(v)\in \bar B(x_0, C(1+C)R),\ F(v)=1\Bigr\}
\end{split}
\end{equation}
where $v^N$ is the normal part of $v$ with respect to the vector $\overrightarrow{\pi(v)z}$ and the scalar product ${<}\cdot,\cdot{>}_{\overrightarrow{\pi(v)z}}$.
If $\eta-\d>0$ then $\SF_\mu$ is strictly convex on $\bar B(x_0, C(1+C)R)$. More precisely, for all $x\in B(x_0, C(1+C)R)$ and for all unit speed geodesic $\g$ starting at $x$, 
\begin{equation}
 \label{59}
(\SF_\mu\circ \g)''(0)\ge \eta-\d.
\end{equation}
\end{prop}
\begin{proof}
 With the notations of section~\ref{Section2}, from~\eqref{47bis} we have 
\begin{equation}
 \label{60}
D_1''(0)=r^{-1}\left(E''(0)-{<}J^T(0),J^T(0){>}_{T(0)}\right).
\end{equation}
Let $\g(s)=c(0,s)$, the unit speed geodesic with initial condition $v=J(0)$, and $f(s)=\SF_\mu(\g(s))$.
Equation~\eqref{60} together with~\eqref{47} gives 
\begin{equation}
 \label{61}
f''(0)\ge -\d+\int_M\sqrt{k}\cot\left(\sqrt{k}\rho(\pi(v),z)\right){<}v^N,v^N{>}_{\overrightarrow{\pi(v)z}}\,\mu(dz)\ge \eta-\d.
\end{equation}
From this we get the condition for the strict convexity of $\SF_\mu$. 

\end{proof}

For $x\in M$ define the measure $\mu_x=\mu-\mu(\{x\})\d_x$. Then the map $y\mapsto \SF_{\mu_x}(y)$ is differentiable at $y=x$. 

Since 
\begin{equation}
\label{33}
\SF_{\mu}(y)=\SF_{\mu_x}(y)+\mu(\{x\})\rho(y,x) 
\end{equation}
and for $v\in T_xM$, $\SF_{\mu}$ is differentiable in the direction $v$ with derivative
\begin{equation}
\label{33bis}
\langle d\SF_{\mu}, v\rangle=\langle d\SF_{\mu_x},v\rangle+\mu(\{x\})F(-v),
\end{equation}
we see that $x$ is a local minimum of $\SF_\mu$ if and only if for all nonzero $v\in T_xM$
\begin{equation}
\label{34}
\mu(\{x\})F(-v)\ge \langle d\SF_{\mu_x},-v\rangle
\end{equation}
which is equivalent to 
\begin{equation}
\label{38}
\mu(\{x\})\ge \f{\left(F^\ast\left(d\SF_{\mu_x}\right)\right)^2}{F\left(\SL^{-1}(d\SF_{\mu_x})\right)}
\end{equation}
(take $\di -v=\f{\SL^{-1}(d\SF_{\mu_x})}{F\left(\SL^{-1}(d\SF_{\mu_x})\right)}$).
But since $F^\ast=F\circ \SL^{-1}$, we get 
\begin{prop}
\label{L1}
A point $x$ in $M$ is a local minimum of $\SF_\mu$ if and only if
\begin{equation}
\label{35}
\mu(\{x\})\ge F^\ast\left(d\SF_{\mu_x}\right).
\end{equation}
\end{prop}
Note that for the Riemannian case this result is due to Le Yang~\cite{Yang:10}.

Define 
the vector
\begin{equation}
\label{31}
H(x)=\grad_y\left(\SF_{\mu_x}(y)\right)|_{y=x}.
\end{equation}
Alternatively, 
\begin{equation}
\label{30}
H(x)=\SL^{-1}\left(\int_{M\backslash\{x\}}\SL\left(-\f1{\rho(x,z)}\overrightarrow{xz}\right)\mu(dz)\right).
\end{equation}
Let $a\mapsto x(a)$ be the path in $M$ defined by $x(0)=x_0$ and 
\begin{equation}
\label{32}
\begin{array}{ccc}
 {\dot x}(a)=&-H(x(a))\ &\hbox{if for all $a'\le a$},\  \mu(\{x(a')\})<F^\ast\left(d\SF_{\mu_{x(a')}}\right) ;\\
 {\dot x}(a)=&0\ &\hbox{if for some $a'\le a$,}\  \mu(\{x(a')\})\ge F^\ast\left(d\SF_{\mu_{x(a')}}\right).
\end{array}
\end{equation}
Define 
\begin{equation}
\label{36}
f(a)=\SF_\mu(x(a)).
\end{equation}
We have for the right derivative of $f$ when $x(a)$ is not a minimal point of $\SF_\mu$:
\begin{align*}
f_+'(a)&=\left\langle d_{x(a)}(\SF_{\mu_{x(a)}}), {\dot x}(a)\right\rangle+\mu(\{x(a)\})F\left(-{\dot x}(a)\right)\\
&=-\left\langle d_{x(a)}(\SF_{\mu_{x(a)}}), \SL^{-1}(d_{x(a)}(\SF_{\mu_{x(a)}}))\right\rangle\\&\quad+\mu(\{x(a)\})F\left(\SL^{-1}(d_{x(a)}(\SF_{\mu_{x(a)}}))\right)\\
&=-F^\ast\left(d_{x(a)}(\SF_{\mu_{x(a)}})\right)^2+\mu(\{x(a)\})F\left(\SL^{-1}(d_{x(a)}(\SF_{\mu_{x(a)}}))\right).
\end{align*}
We get 
\begin{equation}
\label{37}
f_+'(a)=-F^\ast\left(d_{x(a)}(\SF_{\mu_{x(a)}})\right)\left(F^\ast\left(d_{x(a)}(\SF_{\mu_{x(a)}})\right)-\mu(\{x(a)\})\right)
\end{equation}
which is negative as soon as $x(a)$ is not a minimal point of $\SF_\mu$.
From this we get the following
\begin{prop}
\label{P7}
Assume that $\mu$ is supported by a compact ball $\bar B(x_0,R)$, that the support of $\mu$ is not contained in a single geodesic and  that for all $z\in\bar B(x_0,R)$, the forward distance to $z$ is convex, and strictly convex in any geodesic of $\bar B(x_0,C(1+C)R)$ which does not contain $z$. Then the path $a\mapsto x(a)$ converges to the median~$m$ of $\mu$.
\end{prop}
\begin{proof}
Similar to the proof of proposition~\ref{P3}
\end{proof}

\section{An algorithm for computing  $p$-means}\label{Section5}
\setcounter{equation}0

\begin{lemma}
 \label{L4}
Assume $\SK\ge -\b^2$, $\ST\le \d'$, $\SD\le D$ with $\b>0$, $\d'\ge 0$ $D\ge 1$. For $p> 1$, $r>0$, define 
\begin{equation}
 \label{68}
H(r)=H_{p,\b,D,\d'}(r):=pr^{p-2}\left(D^2\max\left((p-1),r\b\coth (r\b)\right)+\d'r\right).
\end{equation}
If  $\mu$ is a probability measure on $M$ with bounded support and $x\in M$, define
\begin{equation}
 \label{69}
H_{\mu}(x)=H_{\mu,p,\b,D,\d'}(x):=\int_MH_{p,\b,D,\d'}(\rho(x,y))\,d\mu.
\end{equation}

If $t\mapsto\g(t)$ is a unit speed geodesic  then for all $t$
\begin{equation}
 \label{70}
(\SE_{\mu,p}\circ \g)''(t)\le H_{\mu}(\g(t)).
\end{equation}
\end{lemma}

\begin{proof}
 For $x,y\in M$, $r=\rho(x,y)$, $s\mapsto \g(s)=c(0,s)$ a unit speed geodesic started at $x=c(0,0)$, $t\mapsto c(t,s)$ the geodesic satisfying $c(1,s)=y$, 
we have
$$
D_p''(0)\le pr^{p-2}\left(D^2\max \left((p-1),r\b\coth (r\b)\right)+\d'r\right).
$$
Integrating with respect to $y$ this
equation
gives the result.
\end{proof}

\begin{remark}
 \label{R4}
If $p\ge 2$ or $\mu$ has a smooth density then the function $H_\mu$ is bounded on all compact sets.
\end{remark}

 The main result is the following (see \cite{Le:04} for a similar result in a Riemannian manifold). 
\begin{prop}
 \label{P10}
Assume $-\b^2\le  \SK\le k$, $-\d\le \ST\le \d'$, $\SC\le C$ and $\SD\le D$ for some $\b, k,\d,\d'>0$ and $C, D\ge 1$. Let $p> 1$.
Assume the support of $\mu$ is contained in $B(x_0,R)$ and $\SE_{\mu,p}$ is strictly convex on $\bar B(x_0, C(C+1)R)$.
Assume furthermore that the function $H_\mu=H_{\mu,p,\b,D,\d'}$ is bounded on $\bar B(x_0, C(C+1)R)$ by a constant $C_H>0$, and that the injectivity radius at any point of $\bar B(x_0, C(C+1)R)$ is larger than $C^2+C+1$.
Define the gradient algorithm as follows: 

\smallbreak

{\bf Step 1}
Start from a point  $x_1\in B(x_0, C(C+1)R)$ such that $\SE_{\mu,p}(x_1)\le R^p$ (take for instance $x_1=x_0$) and let $k=1$. 

{\bf Step 2} 
Let
\begin{equation}
 \label{64}
v_k=\f{\grad(-\SE_{\mu,p}(x_k)))}{F\left(\grad(-\SE_{\mu,p}(x_k))\right)},\qquad\hbox{}\qquad t_k=\f{F\left(\grad(-\SE_{\mu,p}(x_k))\right)}{C_H}.
\end{equation}
and let $\g_k$ be the geodesic satisfying $\g_k(0)=x_k$, $\dot\g_k(0)=v_k$. Define 
\begin{equation}
 \label{67}
x_{k+1}=\g_k(t_k)
\end{equation}
then do again step~2 with $k=k+1$.

\smallbreak

Then the sequence $(x_k)_{k\ge 1}$ converges to $e_p$.
\end{prop}
\begin{proof}
We first prove that the sequence $(\SE_{\mu,p}(x_k))_{k\in \NN}$ is nonincreasing. For this we write 
\begin{equation}
 \label{72}
\begin{split}
\SE_{\mu,p}(\g_k(t_{k}))&\le \SE_{\mu,p}(\g_k(0))+\left\langle d \SE_{\mu,p}, v_k\right\rangle t_{k}+C_H\f{t_{k}^2}{2}\\&
\le \SE_{\mu,p}(\g_k(0))-F\left(\grad(-\SE_{\mu,p}(x_k))\right)\f1{C_H}F\left(\grad(-\SE_{\mu,p}(x_k))\right)\\
&+\f{C_H}{2}\left(\f{F\left(\grad(-\SE_{\mu,p}(x_k))\right)}{C_H}\right)^2\\
&=\SE_{\mu,p}(\g_k(0))-\f{C_H}{2}\left(\f{F\left(\grad(-\SE_{\mu,p}(x_k))\right)}{C_H}\right)^2.
\end{split}
\end{equation}
This proves that the sequence is nonincreasing. As a consequence, for all $k\ge 1$, $x_k\in \bar B(x_0, C(C+1)R)$, since $\SE_{\mu,p}(x_k)\le R^p$ and for all $x\not\in \bar B(x_0, C(C+1)R)$, $\SE_{\mu,p}(x)> R^p$.

Next we prove that $\SE_{\mu,p}(x_k)$ converges to $\SE_{\mu,p}(e_p)$. We know that $\SE_{\mu,p}(x_k)$ converges to $a\ge \SE_{\mu,p}(e_p)$. Extracting a subsequence $x_{k_\ell}$ converging to some $x_\infty\in  \bar B(x_0, C(C+1)R)$, this implies that $t_{k_\ell}$ converges to $0$. But this is possible only if $x_\infty=e_p$, which implies that $a= \SE_{\mu,p}(e_p)$. As a consequence, any converging subsequence has  $e_p$ as a limit, and this implies that $x_k$ converges to $e_p$.
\end{proof}

\begin{remark}
 \label{R5}
For this result we need the Hessian of $\SE_{\mu,p}$ to be bounded, and the subgradient algorithm in Riemannian manifolds as developed in~\cite{Yang:10} does not work. The reason is that for this algorithm, we would need to take  
$$
v_k=\f{\grad_{\overrightarrow{x_ke_p}}(-\SE_{\mu,p}(x_k)))}{F\left(\grad_{\overrightarrow{x_ke_p}}(-\SE_{\mu,p}(x_k))\right)}
$$
where $\grad_{\overrightarrow{x_ke_p}}$ denotes the gradient with respect to the metric ${<}\cdot,\cdot{>}_{\overrightarrow{x_ke_p}}$. So we would need to know $e_p$!
\end{remark}
\begin{cor}
 \label{E1}
Let  $p=2$. If  $\di R\le \f1{C(C+1)^2\sqrt k}\arctan\left(\f{\sqrt k}{C\d^2}\right)$ or $M$ has nonpositive flag curvature, then the algorithm of Proposition~\ref{P10} can be applied with the appropriate constants
\end{cor}
\begin{proof}
With this assumption, by Proposition~\ref{P8} the function
 $\SE_{\mu,2}$ is strictly convex on $\bar B(x_0,C(C+1)R)$. 
 \end{proof}

%
%

\providecommand{\bysame}{\leavevmode\hbox to3em{\hrulefill}\thinspace}


\begin{thebibliography}{10}



\bibitem{Afsari:10}
B. Afsari, \emph{Riemannian $L^p$ center of mass : existence, uniqueness, and convexity}, Proceedings of the American Mathematical Society, S 0002-9939(2010)10541-5, Article electronically published on August 27, 2010.

\bibitem{informationgeometry-2000}
S.I. Amary and H. Nagaoka, \emph{Methods of Information Geometry} Translations of Mathematical Monographs, Vol. 191, AMS, Oxford University Press, 2000.

\bibitem{Alvarez_Paiva:06}
J.C. Alvarez Paiva, \emph{Some problems on Finsler geometry}, Handbook of differential geometry, Vol. II, pp. 1--33, Elsevier/North-Holland, 2006.

\bibitem{Arnaudon-Li:05}
M. Arnaudon and X.M. Li, \emph{Barycenters of measures transported by stochastic flows}, The Annals of Probability, 33 (2005), no. 4, 1509--1543

\bibitem{Astola-Florack:09}
L. Astola and L. Florack, \emph{Finsler geometry on higher oredre tensor fields and applications to high angular resolution diffusion imaging}, Scale Space and Variational Methods in Computer Vision, Second International Conference, SSVM 2009, Voss, Norway, June 1-5, 2009. Proceedings (SSVM), pp. 224--234, 2009

\bibitem{Auslander:55}
L. Auslander, \emph{On curvature in Finsler geometry}, Transactions of the American Mathematical Society; Vol. 79, No.2 (Jul. 1955), pp. 378--388

\bibitem{Bao-Chern-Shen:00}
D. Bao, S.S. Chern and Z. Shen, \emph{An introduction to Riemann-Finsler geometry}, Graduate Texts in Mathematics, Springer, 2000

\bibitem{Chern-Shen:05}
S.S. Chern and Z. Shen, \emph{Riemann-Finsler geometry}, Nankai
Tracts in Mathematics, Vol. 6, World Scientific, 2005

\bibitem{Corcuera-Kendall:99}
J.M. Corcuera and W.S. Kendall, \emph{Riemannian barycentres and geodesic convexity} Math. Proc. Cambridge Philos. Soc.  127  (1999),  no. 2, 253–269.

\bibitem{Elkan:2003}
  Charles Elkan,
 \emph{Using the Triangle Inequality to Accelerate $k$-Means},
  Proceedings of the Twentieth International  Conference  on
Machine Learning (ICML),
  (2003),
  pp. 147-153


\bibitem{Emery-Mokobodzki:91}
M. Emery and G. Mokobodzki, \emph{Sur le barycentre d'une probabilit\'e dans une vari\'et\'e}, S\'eminaire de Probabilit\'es XXV, Lecture Notes in Mathematics 1485 (Springer, Berlin, 1991), pp. 220--233

\bibitem{Fechnerian:1999}
 Ehtibar  Dzhafarov  and Hans Colonius, 
  \emph{Fechnerian metrics in unidimensional and multidimensional
stimulus spaces},
   Psychonomic Bulletin \&  Review, Volume 6, Issue 2
  Springer New York, 1999

\bibitem{Fletcher:09}
P.T. Fletcher, S. Venkatasubramanian, S. Joshi, \emph{The geometric median on Riemannian manifolds with application to robust atlas estimation}, NeuroImage, 45 (2009), pp. S143--S152

\bibitem{FugledeTopsoe:2004}
  Bent Fuglede and Flemming Topsoe,
 \emph{Jensen-Shannon divergence and Hilbert space embedding},
IEEE International Symposium on Information Theory (2004)
  pp. 31--31

\bibitem{Gallego_Torrome:08}
R. Gallego Torrome, \emph{On the generalization of theorems from Riemannian to Finsler geometry I: Metric Theorems}, arXiv:math/0503704v3, [math.DG] 3 Apr 2008


\bibitem{breakdown}
F.R. Hampel, P.J. Rousseeuw, E.M. Ronchetti, W.A. Stahel, \emph{Robust Statistics The Approach Based on Influence Function}, Wiley, New York, 1986

\bibitem{Karcher:77}
H. Karcher, \emph{Riemannian center of mass and mollifier smoothing}, Communications on Pure and Applied Mathematics, vol XXX (1977), 509--541

\bibitem{Kendall:90}
W.S. Kendall, \emph{Probability, convexity and harmonic maps with small image I: uniqueness and fine existence}, Proc. London Math. Soc. (3) 61 no. 2 (1990) pp. 371--406

\bibitem{Kendall:91}
W.S. Kendall, \emph{Convexity and the hemisphere}, J. London Math. Soc., (2) 43 no.3 (1991), pp. 567--576

\bibitem{Kuhn:73}
H.W. Kuhn, \emph{A note on Fermat's problem}, Mathematical Programmming, 4 (1973), pp. 98--107

\bibitem{Le:04}
H. Le, \emph{Estimation of Riemannian barycentres}, LMS J. Comput. Math. 7 (2004), pp. 193--200

\bibitem{MetricSkyline:2009},
Fuhry, David and Jin, Ruoming and Zhang, Donghui,
\emph{Efficient skyline computation in metric space},
 Proceedings of the 12th International Conference on
Extending Database Technology: Advances in Database Technology,
 EDBT '09, ACM, New York, NY, USA
 (2009),
pp. 1042--1051,

\bibitem{MPAT:08}
J. Melonakos, E. Pichon, S. Angenent and A. Tannenbaum, {Finsler active contours}, IEEE Trans. Pattern Anal. Marc. Intell. 30(3), pp. 412--423, 2008.

\bibitem{Ohta:09}
S-i Ohta, \emph{Uniform convexity and smoothness, and their applications in Finsler geometry}, Math. Ann. 343 (2009), pp. 669--699

\bibitem{Nielsen-Nock:09}
F. Nielsen and R. Nock, \emph{Sided and symmetrized Bregman centroids}, IEEE transactions on Information Theory, Vol. 55 Issue 66 (June 2009), pp. 2882--2904

\bibitem{Ostresh:78}
L.M.JR. Ostresh, \emph{On the convergence of a class of iterative methods for solving Weber location problem}, Operation. Research, 26 (1978), no. 4

\bibitem{PKD:09}
M. Péchaud, R. Keriven and M. Descoteaux, \emph{Brain connectivity using geodesics in hardi}, IEEE International Conference on Medical Image Computing and Computed Assisted Intervention (MICCAI), London, Sept. 2009

\bibitem{Picard:94}
J. Picard, \emph{Barycentres et martingales dans les vari\'et\'es}, Ann. Inst. H. Poincar\'e Probab. Statist. 30 (1994), pp. 647--702

\bibitem{Sahib:98}
A. Sahib, \emph{Esp\'erance d'une variable al\'eatoire \`a valeurs dans un espace m\'etrique}, Th\`ese de l'universit\'e de Rouen (1998)

\bibitem{Shen:06}
Z. Shen, \emph{Riemann-Finsler geometry, with applications to information geometry}, Chinese Annals of Mathematics, Series B, 27, pp. 73--94, 2006

\bibitem{Vardi:00}
Vardi, Y. and Zhang, C.H.  \emph{The multivariate $L_1$-median and associated data dept}, Proc. Natl. Acad. Sci, USA, 97, (2000), pp.1423-1426.

\bibitem{Weiszfeld:37}
E. Weiszfeld, \emph{Sur le point pour lequel la somme des distances de $n$ points donn\'es est minimum}, T\^ohoku Math. J. 43 (1937), pp. 355--386

\bibitem{Wu-Xin:07}
B.Y. Wu and Y.L. Xin, \emph{Comparison theorems in Finsler geometry and their applications}, Math. Ann. 337 (2007), pp. 177--196

\bibitem{Yang:10}
L.Yang, \emph{Riemannian median and its estimation}, to appear in LMS Journal of Computation and Mathematics

\bibitem{ZSN:09}
C.Zach, L. Shan and M. Niethammer, \emph{Globally Optimal Finsler Active Contours}, Lecture Notes in Computer Science, Vol. 5748 (2009), pp. 552--561




\end{thebibliography}
\end{document}